\documentclass[CEJP,DVI]{my_cej}

\usepackage{layout}
\usepackage{amsmath}
\usepackage{textcomp}
\usepackage{hyperref}

\usepackage{amsmath,amssymb}

\title{Noether's theorem for fractional variational problems of variable order}

\articletype{Research Article}

\author{Tatiana~Odzijewicz\inst{1}$^,$\email{tatianao@ua.pt},
        Agnieszka~B.~Malinowska\inst{2}$^,$\email{a.malinowska@pb.edu.pl},
        Delfim~F.~M.~Torres\inst{1}$^,$\email{delfim@ua.pt}}

\institute{
     \inst{1} CIDMA -- Center for Research and Development in Mathematics and Applications,\\
     Department of Mathematics, University of Aveiro, 3810-193 Aveiro, Portugal
     \inst{2} Faculty of Computer Science, Bia{\l}ystok University of Technology, 15-351 Bia\l ystok, Poland
}

\abstract{We prove a necessary optimality condition of Euler--Lagrange type for fractional variational
problems with derivatives of incommensurate variable order. This allows us to state a version
of Noether's theorem without transformation of the independent (time) variable.
Considered derivatives of variable order are defined in the sense of Caputo.}



\begin{document}

\maketitle


\section{Introduction}

Fractional calculus is the scientific discipline that deals
with integrals and derivatives of arbitrary real (or complex) order.
Since the XVIIth century, when fractional integration and differentiation
was brought up for the first time, many well-known mathematicians contributed
to the theory, among them Euler, Laplace, Fourier, Abel, Liouville and Riemann.
For a comprehensive knowledge of fractional calculus
we refer the reader to the books \cite{book:Kilbas,book:Klimek,book:Podlubny,Samko:book}.

The XXth century has shown that fractional order calculus is more adequate
to describe real world problems than the integer/standard
calculus \cite{Caponetto,MR2969872}. Therefore, not only mathematicians
have currently a strong interest in the fractional calculus
but also researchers in applied fields such as mechanics, physics, chemistry,
biology, economics, control theory and signal processing \cite{b:Baleanu,book:Ortigueira,TM:rec}.

A generalization of the fractional calculus was proposed in 1993
by Samko and Ross \cite{Samko:Ross2}. They considered integrals and derivatives of order $\alpha$,
where $\alpha$ is not a constant but a function. Afterwards, several works were dedicated
to these operators \cite{Samko:Ross1,Samko}. Interesting applications of the proposed calculus are found
in the theory of viscous flows and mechanics \cite{Coimbra,Diaz:Coimbra,Lorentzo,Ramirez:Coimbra1,Ramirez:Coimbra2}.

In 1996, Riewe initiated the theory of the fractional calculus of variations by considering
problems of mechanics with fractional order derivatives \cite{CD:Riewe:1996,CD:Riewe:1997}.
Nowadays, the fractional calculus of variations is under current strong research (see, e.g.,
\cite{Ref:2,DerInt,Ref:1,Ref:4,MyID:152,Cresson,gastao4,Ref:3,comDorota,MyID:181}
and references therein). For the state of the art on the subject we refer to the recent
book \cite{TheBook}. Here we remark that results for problems of the calculus of variations
with variable order fractional operators are scarce, reducing to those found
in \cite{Atanackovic1,MyID:241,Tatiana:IDOTA2011}. In particular,
no Noether's symmetry theorem of variable order has been proved before.

The main aim of the current work is to generalize the Noether theorem,
originally proved by the German mathematician Emmy Noether in 1918,
asserting that invariance properties of integral functionals lead to
corresponding conservation laws \cite{jurgen,book:Kossman,book:vanBrunt}.
In contrast with previous works \cite{MR2524365,gastao1,gastao3,gastao4},
where fractional versions of the Noether theorem for a constant non-integer order $\alpha$
are obtained, here we consider more general fractional variational problems
with variable order derivatives.

The paper is organized as follows. In Section~\ref{sec:prelim}, basic definitions
and properties of variable order fractional operators are given. Sections~\ref{sec:ibp},
\ref{sec:MR} and \ref{subsec:noether} contain our main results:
we prove integration by parts formulas for variable order fractional integrals
(Theorem~\ref{thm:IntByParts1}) and derivatives (Theorem~\ref{thm:IntByParts2}),
a necessary optimality condition of Euler--Lagrange type
for fractional variational problems with derivatives of incommensurate variable order
(Theorem~\ref{theorem:EL}), and a version of Noether's theorem for such fractional variable order problems
(Theorem~\ref{theorem:Noether}). Then, in Section~\ref{sec:ex}, we illustrate our results through an example.
We finish with Section~\ref{sec:conc} of conclusion.


\section{Variable order fractional operators}
\label{sec:prelim}

In this section we recall the basic definitions necessary in the sequel,
and derive a new interesting mapping property for the fractional integral operators
of variable order (Theorem~\ref{thm:fpr}).
We set $\Delta := \{(t,\tau)\in \mathbb{R}^2 : a \le \tau < t \le b\}$.

\begin{definition}[Left and right Riemann--Liouville integrals of variable order]
Let $\alpha : \Delta \rightarrow (0,1)$ and $f\in L_1 [a,b]$. Then,
\begin{equation*}
{_{a}}\textsl{I}^{\alpha(\cdot,\cdot)}_{t}f(t)
= \int\limits_a^t\frac{1}{\Gamma(\alpha(t,\tau))}
(t-\tau)^{\alpha(t,\tau)-1}f(\tau)d\tau, \quad t>a,
\end{equation*}
is called the left Riemann--Liouville integral
of variable fractional order $\alpha(\cdot,\cdot)$, while
\begin{equation*}
{_{t}}\textsl{I}^{\alpha(\cdot,\cdot)}_{b}f(t)
=\int\limits_t^b \frac{1}{\Gamma(\alpha(\tau,t))}
(\tau-t)^{\alpha(\tau,t)-1}f(\tau)d\tau, \quad t<b,
\end{equation*}
denotes the right Riemann--Liouville integral
of variable fractional order $\alpha(\cdot,\cdot)$.
\end{definition}

\begin{theorem}
\label{thm:fpr}
Let $\alpha(t,\tau)=\alpha(t-\tau)$ with
$0<\alpha(t-\tau)<1-\frac{1}{n}$ for all $0 < t - \tau \le b$
and a certain $n\in\mathbb{N}$ greater or equal than two.
If $f\in AC[0,b]$, then
\begin{equation*}
{_{0}}\textsl{I}^{1-\alpha(\cdot,\cdot)}_{t}f(t)
= \int\limits_0^t\frac{1}{\Gamma(1-\alpha(t-\tau))}
(t-\tau)^{-\alpha(t-\tau)}f(\tau)d\tau \in AC[0,b].
\end{equation*}
If $f\in AC[-b,0]$, then
\begin{equation*}
{_{t}}\textsl{I}^{1-\alpha(\cdot,\cdot)}_{0}f(t)
= \int\limits_t^0\frac{1}{\Gamma(1-\alpha(\tau-t))}
(\tau-t)^{-\alpha(\tau-t)}f(\tau)d\tau \in AC[-b,0].
\end{equation*}
\end{theorem}

\begin{proof}
We give the proof for the left integral; the other case being proved similarly.
Let $k(s):=\frac{1}{\Gamma(1-\alpha(s))}s^{-\alpha(s)}$. Since $0<\alpha(s)<1-\frac{1}{n}$,
\begin{enumerate}
\item for $s\leq 1$ we have $\ln s\geq 0$ and $s^{-\alpha(s)}< 1$,
\item for $s< 1$ we have $\ln s< 0$ and $s^{-\alpha(s)}< s^{\frac{1}{n}-1}$.
\end{enumerate}
Therefore,
\begin{equation*}
\int_0^b\left| k(s) \right|d\tau=\int_0^b\left| \frac{1}{\Gamma(1-\alpha(s))}s^{-\alpha(s)} \right|d\tau\\
<\int_0^1\frac{1}{\Gamma(1-\alpha(s))}ds+\int_1^b\frac{1}{\Gamma(1-\alpha(s))}s^{\frac{1}{n}-1}ds.
\end{equation*}
Using the inequality
\begin{equation}
\label{eq:ineq:gam}
\Gamma(x+1) \geq \frac{x^2+1}{x+1},
\end{equation}
valid for all $x \in [0,1]$ (see \cite{Ivady}), we obtain that
\begin{equation*}
\int_0^1\frac{1}{\Gamma(1-\alpha(s))}ds+\int_1^b\frac{1}{\Gamma(1-\alpha(s))}s^{\frac{1}{n}-1}ds
<\int_0^1 ds+\int_1^b s^{\frac{1}{n}-1}ds=1+nb^{\frac{1}{n}}-n<\infty.
\end{equation*}
It means that $k(s) \in L_1[0,b]$.
Moreover, the condition $f\in AC\left([0,b]\right)$ implies that
\begin{equation*}
f(t)=\int_0^t g(\tau)d\tau+f(0),~ \textnormal{where}~g\in L_1\left([0,b]\right).
\end{equation*}
Define
$$
h(t):=\int_0^t \frac{1}{\Gamma(1-\alpha(s))}s^{-\alpha(s)}g(t-s)ds
$$
and integrate:
\begin{equation}
\label{eq:order}
\int_0^t h(\theta)d\theta =\int_0^t d\theta \int_0^{\theta} \frac{1}{\Gamma(1-\alpha(s))}s^{-\alpha(s)}g(\theta-s)ds.
\end{equation}
Observe that
\begin{equation*}
\begin{split}
\int_0^t\left(\int_s^t \left|\frac{1}{\Gamma(1-\alpha(s))}s^{-\alpha(s)}g(\theta-s)\right|d\theta\right)ds
&=\int_0^t\left(\left|\frac{1}{\Gamma(1-\alpha(s))}s^{-\alpha(s)}\right|\left(
\int_s^t\left|g(\theta-s)\right|d\theta\right)\right)ds\\
&<\left\|g\right\|_1\int_0^b\left|\frac{1}{\Gamma(1-\alpha(s))}s^{-\alpha(s)}\right|ds\\
&<\infty.
\end{split}
\end{equation*}
It means that assumptions of Fubini's theorem are satisfied,
so we can change an order of integration in \eqref{eq:order}. Hence,
$$
\int_0^t h(\theta)d\theta =\int_0^t ds \int_s^t \frac{1}{\Gamma(1-\alpha(s))}s^{-\alpha(s)}g(\theta-s)d\theta
=\int_0^t \frac{1}{\Gamma(1-\alpha(s))}s^{-\alpha(s)}ds \int_s^t g(\theta-s)d\theta.
$$
Putting $\xi=\theta-s$, one has $d\xi=d\theta$ and we get
$$
\int_0^t h(\theta)d\theta=\int_0^t \frac{1}{\Gamma(1-\alpha(s))}s^{-\alpha(s)}ds \int_0^{t-s} g(\xi)d\xi.
$$
But $\int_0^{t-s} g(\xi)d\xi=f(t-s)-f(0)$ and, therefore, the following equality holds:
$$
\int_0^t h(\theta)d\theta=\int_0^t \frac{1}{\Gamma(1-\alpha(s))}s^{-\alpha(s)}f(t-s)ds
-f(0)\int_0^t \frac{1}{\Gamma(1-\alpha(s))}s^{-\alpha(s)} ds.
$$
Putting $\tau=t-s$, $d\tau=-ds$ and
$$
\int_0^t h(\theta)d\theta=\int_0^t \frac{1}{\Gamma(1-\alpha(t-\tau))}(t-\tau)^{-\alpha(t-\tau)}f(\tau)d\tau
-f(0)\int_0^t \frac{1}{\Gamma(1-\alpha(t-\tau))}(t-\tau)^{-\alpha(t-\tau)} d\tau.
$$
Both functions on the right-hand side of the equality
$$
\int_0^t \frac{1}{\Gamma(1-\alpha(t-\tau))}(t-\tau)^{-\alpha(t-\tau)}f(\tau)d\tau
=\int_0^t h(\theta)d\theta+f(0)\int_0^t \frac{1}{\Gamma(1-\alpha(t-\tau))}(t-\tau)^{-\alpha(t-\tau)} d\tau
$$
belong to $AC\left([0,b]\right)$, hence ${_{0}}\textsl{I}^{\alpha(\cdot-\cdot)}_{t}f\in AC[0,b]$.
\end{proof}

\begin{definition}[Left and right Riemann--Liouville derivatives of variable order]
Let $\alpha : \Delta \rightarrow (0,1)$.
If ${_{a}}\textsl{I}^{1-\alpha(\cdot,\cdot)}_{t}f \in AC[a,b]$,
then the left Riemann--Liouville derivative of variable fractional
order $\alpha(\cdot,\cdot)$ is defined by
\begin{equation*}
{_{a}}\textsl{D}^{\alpha(\cdot,\cdot)}_{t} f(t)
= \frac{d}{dt} {_{a}}\textsl{I}^{1-\alpha(\cdot,\cdot)}_{t} f(t)
=\frac{d}{dt}\int\limits_a^t
\frac{1}{\Gamma(1-\alpha(t,\tau))}(t-\tau)^{-\alpha(t,\tau)} f(\tau)d\tau, \quad t>a,
\end{equation*}
while the right Riemann--Liouville derivative of variable order
$\alpha(\cdot,\cdot)$ is defined for functions $f$ such that
${_{t}}\textsl{I}^{1-\alpha(\cdot,\cdot)}_{b}f\in AC[a,b]$ by
\begin{equation*}
{_{t}}\textsl{D}^{\alpha(\cdot,\cdot)}_{b}f(t)
= -\frac{d}{dt} {_{t}}\textsl{I}^{1-\alpha(\cdot,\cdot)}_{b}f(t)
=\frac{d}{dt}\int\limits_t^b
\frac{-1}{\Gamma(1-\alpha(\tau,t))}(\tau-t)^{-\alpha(\tau,t)}f(\tau)d\tau, \quad t<b.
\end{equation*}
\end{definition}

\begin{definition}[Left and right Caputo derivatives of variable fractional order]
\label{definition:Caputo}
Let $0<\alpha(t,\tau)<1$ for all $t, \tau \in [a,b]$.
If $f\in AC[a,b]$, then the left Caputo derivative
of variable fractional order $\alpha(\cdot,\cdot)$ is defined by
\begin{equation*}
{^{C}_{a}}\textsl{D}^{\alpha(\cdot,\cdot)}_{t}f(t)
=\int\limits_a^t
\frac{1}{\Gamma(1-\alpha(t,\tau))}(t-\tau)^{-\alpha(t,\tau)}\frac{d}{d\tau}f(\tau)d\tau, \quad t>a,
\end{equation*}
while the right Caputo derivative of variable fractional order $\alpha(\cdot,\cdot)$ is given by
\begin{equation*}
{^{C}_{t}}\textsl{D}^{\alpha(\cdot,\cdot)}_{b}f(t)
=\int\limits_t^b\frac{-1}{\Gamma(1-\alpha(\tau,t))}
(\tau-t)^{-\alpha(\tau,t)}\frac{d}{d\tau}f(\tau)d\tau, \quad t<b.
\end{equation*}
\end{definition}


\section{Variable order fractional integration by parts}
\label{sec:ibp}

In this section we derive integration by parts formulas,
which are essential for proving Euler--Lagrange equations.

\begin{theorem}[Integration by parts for variable order fractional integrals]
\label{thm:IntByParts1}
If $\frac{1}{n}<\alpha(t,\tau)<1$ for all $(t,\tau) \in \Delta$
and a certain $n\in\mathbb{N}$ greater or equal than two,
and $f,g\in C\left([a,b];\mathbb{R}\right)$, then
\begin{equation*}
\int_a^b g(t){_{a}}\textsl{I}^{\alpha(\cdot,\cdot)}_{t} f(t)dt
= \int_a^b f(t){_{t}}\textsl{I}^{\alpha(\cdot,\cdot)}_{b} g(t)dt.
\end{equation*}
\end{theorem}

\begin{proof}
Define
\[
F(\tau,t):= \left\{
\begin{array}{ll}
\left|\frac{1}{\Gamma(\alpha(t,\tau))}(t-\tau)^{\alpha(t,\tau)-1}
g(t)f(\tau)\right| & \mbox{if $\tau < t$},\\
0 & \mbox{if $\tau \geq t$,}
\end{array} \right.
\]
for all $(\tau,t)\in [a,b]\times [a,b]$.
Since $f$ and $g$ are continuous functions on $[a,b]$,
they are bounded on $[a,b]$, \textrm{i.e.},
there exist $C_1,C_2>0$ such that $\left|g(t)\right|\leq C_1$
and $\left|f(t)\right|\leq C_2$, $t\in[a,b]$. Therefore,
\begin{equation*}
\begin{split}
\int_a^b\left(\int_a^b F(\tau,t)d\tau\right)dt
&=\int_a^b \left(\int_a^t\left|\frac{1}{\Gamma(\alpha(t,\tau))}(t-\tau)^{\alpha(t,\tau)-1}
g(t)f(\tau)\right|d\tau\right)dt\\
&\leq C_1 C_2\int_a^b \left(\int_a^t\left|\frac{1}{\Gamma
\left(\alpha(t,\tau)\right)}(t-\tau)^{\alpha(t,\tau)-1}\right|d\tau\right)dt\\
&=C_1 C_2 \int_a^b \left(\int_a^t \frac{1}{\Gamma
(\alpha(t,\tau))} (t-\tau)^{\alpha(t,\tau)-1}d\tau\right)dt .
\end{split}
\end{equation*}
Because $\frac{1}{n}<\alpha(t,\tau)<1$,
\begin{enumerate}

\item for $1\leq t-\tau $ we have $\ln(t-\tau)\geq 0$ and
$(t-\tau)^{\alpha(t,\tau)-1}<1$;

\item for $1>t-\tau$ we have $\ln(t-\tau)<0$ and
$(t-\tau)^{\alpha(t,\tau)-1}<(t-\tau)^{\frac{1}{n}-1}$.

\end{enumerate}
Therefore,
\begin{multline*}
C_1 C_2\int_a^b  \left(\int_a^t \frac{1}{\Gamma
\left(\alpha(t,\tau)\right)} (t-\tau)^{\alpha(t,\tau)-1}d\tau\right) dt\\
<C_1 C_2\int_a^b\left(\int_a^{t-1}\frac{1}{\Gamma
\left(\alpha(t,\tau)\right)} d\tau
+\int_{t-1}^t \frac{1}{\Gamma\left(\alpha(t,\tau)\right)}
(t-\tau)^{\frac{1}{n}-1} d\tau\right)dt.
\end{multline*}
Moreover, by inequality \eqref{eq:ineq:gam},
valid for $x \in [0,1]$, one has
\begin{equation*}
\begin{split}
C_1 C_2 \int_a^b & \left(\int_a^{t-1}\frac{1}{\Gamma
\left(\alpha(t,\tau)\right)} d\tau
+\int_{t-1}^t \frac{1}{\Gamma
\left(\alpha(t,\tau)\right)}(t-\tau)^{\frac{1}{n}-1} d\tau\right)dt\\
&\leq C_1 C_2\int_a^b\left(\int_a^{t-1}\frac{\alpha^2(t,\tau)
+\alpha(t,\tau)}{\alpha^2(t,\tau)+1} d\tau
+\int_{t-1}^t \frac{\alpha^2(t,\tau)+\alpha(t,\tau)}{\alpha^2(t,\tau)
+1}(t-\tau)^{\frac{1}{n}-1} d\tau\right)dt\\
&< C_1 C_2\int_a^b\left(\int_a^{t-1}  d\tau
+\int_{t-1}^t (t-\tau)^{\frac{1}{n}-1} d\tau\right)dt\\
&=C_1 C_2(b-a)\left(\frac{b+a}{2}-1+n-a\right)\\
&<\infty.
\end{split}
\end{equation*}
Hence, one can use the Fubini theorem to change the order
of integration:
\begin{equation*}
\begin{split}
\int_a^b g(t){_{a}}\textsl{I}^{\alpha(\cdot,\cdot)}_{t} f(t)dt
&=\int_a^b \left(\int_a^t g(t)f(\tau)\frac{1}{\Gamma
\left(\alpha(t,\tau)\right)}(t-\tau)^{\alpha(t,\tau)-1}d\tau\right)dt\\
&=\int_a^b \left(\int_\tau^b g(t)f(\tau)\frac{1}{\Gamma
\left(\alpha(t,\tau)\right)}(t-\tau)^{\alpha(t,\tau)-1}dt\right)d\tau\\
&=\int_a^b f(\tau){_{\tau}}\textsl{I}^{\alpha(\cdot,\cdot)}_{b} g(\tau)d\tau.
\end{split}
\end{equation*}
\end{proof}

\begin{theorem}[Integration by parts for variable order fractional derivatives]
\label{thm:IntByParts2}
Let $0<\alpha(t,\tau)<1-\frac{1}{n}$ for all $(t,\tau)\in \Delta$
and a certain $n\in\mathbb{N}$ greater or equal than two.
If $f\in C^1\left([a,b];\mathbb{R}\right)$,
$g\in C\left([a,b];\mathbb{R}\right)$,
and ${_{t}}\textsl{I}^{1-\alpha(\cdot,\cdot)}_{b} g$, then
\begin{equation*}
\int_a^b g(t){^{C}_{a}}\textsl{D}^{\alpha(\cdot,\cdot)}_{t}f(t)dt
=\left.f(t){_{t}}\textsl{I}^{1-\alpha(\cdot,\cdot)}_{b} g(t)\right|_a^b
+\int_a^b f(t){_{t}}\textsl{D}^{\alpha(\cdot,\cdot)}_{b}g(t)dt.
\end{equation*}
If $f\in C^1\left([a,b];\mathbb{R}\right)$,
$g\in C\left([a,b];\mathbb{R}\right)$,
and ${_{a}}\textsl{I}^{1-\alpha(\cdot,\cdot)}_{t} g \in AC[a,b]$, then
\begin{equation*}
\int_a^b g(t){^{C}_{t}}\textsl{D}^{\alpha(\cdot,\cdot)}_{b}f(t)dt
=-\left.f(t){_{a}}\textsl{I}^{1-\alpha(\cdot,\cdot)}_{t} g(t)\right|_a^b
+\int_a^b f(t){_{a}}\textsl{D}^{\alpha(\cdot,\cdot)}_{t}g(t)dt.
\end{equation*}
\end{theorem}

\begin{proof}
By Definition~\ref{definition:Caputo}, it follows that
${^{C}_{a}}\textsl{D}^{\alpha(\cdot,\cdot)}_{t}f(t)
={_{a}}\textsl{I}^{1-\alpha(\cdot,\cdot)}_{t}\frac{d}{dt} f(t)$.
Applying Theorem~\ref{thm:IntByParts1} and integration
by parts for classical (integer order) derivatives, we obtain
\begin{equation*}
\begin{split}
\int_a^b g(t){^{C}_{a}}\textsl{D}^{\alpha(\cdot,\cdot)}_{t}f(t)dt
&=\int_a^b g(t){_{a}}\textsl{I}^{1-\alpha(\cdot,\cdot)}_{t}\frac{d}{dt}f(t)dt
=\int_a^b \frac{d}{dt}f(t){_{t}}\textsl{I}^{1-\alpha(\cdot,\cdot)}_{b}g(t)dt\\
&=\left.f(t){_{t}}\textsl{I}^{1-\alpha(\cdot,\cdot)}_{b} g(t)\right|_a^b
-\int_a^b f(t)\frac{d}{dt}{_{t}}\textsl{I}^{1-\alpha(\cdot,\cdot)}_{b}g(t)dt\\
&=\left.f(t){_{t}}\textsl{I}^{1-\alpha(\cdot,\cdot)}_{b} g(t)\right|_a^b
+\int_a^b f(t){_{t}}\textsl{D}^{\alpha(\cdot,\cdot)}_{b}g(t)dt.
\end{split}
\end{equation*}
The second formula is proved in a similar way.
\end{proof}


\section{Euler--Lagrange equations for incommensurate fractional variational problems of variable order}
\label{sec:MR}

Now we prove necessary optimality conditions of Euler--Lagrange type
for incommensurate order fractional variational problems.
Let $\alpha_i(t,\tau)$, $\beta_i(t,\tau)$, $i=1,\dots,n$, $n\in\mathbb{N}$,
satisfy the assumptions of Theorem~\ref{thm:IntByParts2}.
Consider the following problem.
\begin{problem}
\label{problem:Fundamental}
Find a function $q=q(t)$ for which the functional
\begin{equation}
\label{eq:Functional}
\mathcal{J}[q]=\int\limits_a^b
L\left(t,q(t),{^{C}_{a}}\textsl{D}^{\alpha_1(t,\tau)}_{t}q(t),
\dots,{^{C}_{a}}\textsl{D}^{\alpha_n(t,\tau)}_{t}q(t),
{^{C}_{t}}\textsl{D}^{\beta_1(t,\tau)}_{b}q(t),
\dots,{^{C}_{t}}\textsl{D}^{\beta_n(t,\tau)}_{b}q(t)\right)dt
\end{equation}
attains an extremum on the set
\begin{equation*}
\mathcal{D}=\left\{q\in C^1 ([a,b];\mathbb{R}):q(a)=q_a, q(b)=q_b
\text{ and }{^{C}_{a}}\textsl{D}^{\alpha_i(t,\tau)}_{t}q,
{^{C}_{t}}\textsl{D}^{\beta_i(t,\tau)}_{b}q \in C([a,b];\mathbb{R}),
i=1,\dots,n \right\}.
\end{equation*}
\end{problem}
We assume that $L\in C^1\left([a,b]\times \mathbb{R}^{2n+1};\mathbb{R}\right)$;
$t\mapsto\partial_{i+2}L$ is continuous, has absolutely continuous integral
${_{t}}\textsl{I}^{1-\alpha_i(t,\tau)}_{b}$ and continuous derivative
${_{t}}\textsl{D}^{\alpha_i(t,\tau)}_{b}$ for each $i=1,\dots,n$;
$t\mapsto\partial_{n+i+2}L$ is continuous, has absolutely continuous integral
${_{a}}\textsl{I}^{1-\beta_i(t,\tau)}_{t}$ and continuous derivative
${_{a}}\textsl{D}^{\beta_i(t,\tau)}_{t}$ for each $i=1,\dots,n$.
For simplicity of notation, we introduce the following notation:
\begin{equation*}
\left\{q,\alpha,\beta\right\}(t)
:=\left(t,q(t),{^{C}_{a}}\textsl{D}^{\alpha(t,\tau)}_{t}q(t),
{^{C}_{t}}\textsl{D}^{\beta(t,\tau)}_{b}q(t)\right),
\end{equation*}
where
\begin{equation*}
{^{C}_{a}}\textsl{D}^{\alpha(t,\tau)}_{t}
:= \left({^{C}_{a}}\textsl{D}^{\alpha_1(t,\tau)}_{t},
\dots,{^{C}_{a}}\textsl{D}^{\alpha_n(t,\tau)}_{t}\right),
\quad
{^{C}_{t}}\textsl{D}^{\beta(t,\tau)}_{b}
:=\left({^{C}_{t}}\textsl{D}^{\beta_1(t,\tau)}_{b},
\dots,{^{C}_{t}}\textsl{D}^{\beta_n(t,\tau)}_{b}\right).
\end{equation*}

\begin{definition}
\label{definition:OpProduct}
Let $f\in AC([a,b])$. For ${_{t}}\textsl{I}^{1-\gamma(t,\tau)}_{b} g\in AC([a,b])$
we define the following operator:
\begin{equation}
\label{op1}
\textsl{D}^{\gamma(t,\tau)}_{-}[f,g]:=-f{_{t}}\textsl{D}^{\gamma(t,\tau)}_{b}g
+g{^{C}_{a}}\textsl{D}^{\gamma(t,\tau)}_{t}[f],
\end{equation}
while for ${_{a}}\textsl{I}^{1-\gamma(t,\tau)}_{t} g\in AC([a,b])$ we define
\begin{equation}
\label{op2}
\textsl{D}^{\gamma(t,\tau)}_{+}[f,g]
:=-f{_{a}}\textsl{D}^{\gamma(t,\tau)}_{t}[g]
+g{^{C}_{t}}\textsl{D}^{\gamma(t,\tau)}_{b}[f].
\end{equation}
\end{definition}

\begin{remark}
If $\gamma(t,\tau) \equiv \gamma$ is a constant function, then
\begin{equation*}
\lim\limits_{\gamma\rightarrow 1^{-}}\textsl{D}^{\gamma}_{-}[f,g]
=fg'+gf'=\frac{d}{dt}(fg)=\lim\limits_{\gamma\rightarrow 1^{-}}\textsl{D}^{\gamma}_{-}[g,f]
\end{equation*}
and
\begin{equation*}
\lim\limits_{\gamma\rightarrow 1^{-}}\textsl{D}^{\gamma}_{+}[f,g]
=-fg'-gf'=-\frac{d}{dt}(fg)
=\lim\limits_{\gamma\rightarrow 1^{-}}\textsl{D}^{\gamma}_{+}[g,f].
\end{equation*}
\end{remark}

\begin{theorem}
\label{theorem:EL}
Let function $q$ be a solution to Problem~\ref{problem:Fundamental}. Then,
\begin{equation}
\label{eq:EL}
\partial_2 L\left\{q,\alpha,\beta\right\}(t)
-\sum\limits_{i=1}^n \textsl{D}^{\alpha_i(t,\tau)}_{-}\left[1,
\partial_{i+2}L\left\{q,\alpha,\beta\right\}(t)\right]\\
-\sum\limits_{i=1}^n \textsl{D}^{\beta_i(t,\tau)}_{+}\left[1,
\partial_{n+2+i} L\left\{q,\alpha,\beta\right\}(t)\right]=0.
\end{equation}
\end{theorem}

\begin{proof}
Assume that $q$ is an extremizer of $\mathcal{J}$. Consider the value of
$\mathcal{J}$ at nearby function $\hat{q}(t)=q(t)+\varepsilon\eta(t)$,
where $\varepsilon\in\mathbb{R}$ is a small parameter and
$\eta\in C^1([a,b];\mathbb{R})$ is an arbitrary function satisfying
$\eta(a)=\eta(b)=0$ and such that ${^{C}_{a}}\textsl{D}^{\alpha_i(t,\tau)}_{t}\hat{\eta}$
and ${^{C}_{t}}\textsl{D}^{\beta_i(t,\tau)}_{b}\hat{\eta}$ are continuous. Let
$$
J(\varepsilon)=\mathcal{J}[\hat{q}]
=\int_a^b L\left\{\hat{q},\alpha,\beta\right\}(t) dt.
$$
A necessary condition for $\hat{q}$ to be an extremizer is given by
\begin{multline}
\label{eq:1}
\left.\frac{dJ}{d\varepsilon}\right|_{\varepsilon=0} = 0
\Leftrightarrow\int_a^b\left(\partial_2 L\left\{q,\alpha,\beta\right\}(t) \eta(t)
+\sum\limits_{i=1}^n\partial_{i+2}L\left\{q,\alpha,
\beta\right\}(t){^{C}_{a}}\textsl{D}^{\alpha_i(t,\tau)}_{t}\eta(t)\right.\\
\left.+\sum\limits_{i=1}^n\partial_{n+2+i}L\left\{q,\alpha,
\beta\right\}(t){^{C}_{t}}\textsl{D}^{\beta_i(t,\tau)}_{b}\eta(t)\right)dt=0.
\end{multline}
Using the variable order fractional integration
by parts formulas (Theorem~\ref{thm:IntByParts2}), we obtain that
\begin{multline*}
\int_a^b \partial_{i+2}
L\left\{q,\alpha,\beta\right\}(t){^{C}_{a}}\textsl{D}^{\alpha_i(t,\tau)}_{t}\eta(t) dt\\
=\left.\eta(t) {_{t}}\textsl{I}^{1-\alpha_i(t,\tau)}_{b}\partial_{i+2}
L\left\{q,\alpha,\beta\right\}(t)\right|_a^b
+\int\limits_a^b\eta(t){_{t}}\textsl{D}^{\alpha_i(t,\tau)}_{b}\partial_{i+2}
L\left\{q,\alpha,\beta\right\}(t) dt, \quad i=1,\dots,n,
\end{multline*}
and
\begin{multline*}
\int_a^b \partial_{n+i+2}
L\left\{q,\alpha,\beta\right\}(t){^{C}_{t}}\textsl{D}^{\beta_i(t,\tau)}_{b}\eta(t) dt\\
=-\left.\eta(t) {_{a}}\textsl{I}^{1-\beta_i(t,\tau)}_{t}\partial_{n+i+2}
L\left\{q,\alpha,\beta\right\}(t)\right|_a^b
+\int\limits_a^b\eta(t){_{a}}\textsl{D}^{\beta_i(t,\tau)}_{t}\partial_{n+i+2}
L\left\{q,\alpha,\beta\right\}(t) dt, \quad i=1,\dots,n.
\end{multline*}
Because $\eta(a)=\eta(b)=0$, \eqref{eq:1} simplifies to
\begin{equation*}
\int_a^b \eta(t)\left(\partial_2 L\left\{q,\alpha,\beta\right\}(t)
+\sum\limits_{i=1}^n{_{t}}\textsl{D}^{\alpha_i(t,\tau)}_{b}\partial_{i+2}
L\left\{q,\alpha,\beta\right\}(t)
+\sum\limits_{i=1}^n{_{a}}\textsl{D}^{\beta_i(t,\tau)}_{t}\partial_{n+i+2}
L\left\{q,\alpha,\beta\right\}(t)\right)dt=0.
\end{equation*}
Applying the fundamental lemma of the calculus of variations
(see, e.g., \cite{book:vanBrunt}), we obtain that
\begin{equation*}
\partial_2 L\left\{q,\alpha,\beta\right\}(t)
+\sum\limits_{i=1}^n
{_{t}}\textsl{D}^{\alpha_i(t,\tau)}_{b}\partial_{i+2}
L\left\{q,\alpha,\beta\right\}(t)\\
+\sum\limits_{i=1}^n{_{a}}\textsl{D}^{\beta_i(t,\tau)}_{t}\partial_{n+i+2}
L\left\{q,\alpha,\beta\right\}(t)=0.
\end{equation*}
Finally, by Definition~\ref{definition:OpProduct}, we arrive to \eqref{eq:EL}.
\end{proof}


\section{Noether's theorem of fractional variable order}
\label{subsec:noether}

Now we show, employing operators \eqref{op1} and \eqref{op2},
that invariance of \eqref{eq:Functional} conducts to a variable order version
of a fractional conservation law. To prove this, we borrow the method from
\cite{gastao1,gastao4,jurgen}, where Noether's theorem is stated
for fractional derivatives of constant order.

\begin{definition}
A function $q$ that is a solution to \eqref{eq:EL}
is said to be a variable order fractional extremal
for functional $\mathcal{J}$.
\end{definition}

\begin{definition}
\label{definition:InvariantFunct}
We say that functional \eqref{eq:Functional} is invariant
under an $\varepsilon$-parameter group of infinitesimal transformations
\begin{equation}
\label{eq:transformation}
\bar{q}(t)=q(t)+\varepsilon\xi(t,q(t))+o(\varepsilon)
\end{equation}
if
\begin{equation}
\label{eq:InvarianceCond}
\int_{t_a}^{t_b}L\left(t,q(t),{^{C}_{a}}\textsl{D}^{\alpha(t,\tau)}_{t}q(t),
{^{C}_{t}}\textsl{D}^{\beta(t,\tau)}_{b}q(t)\right)dt\\
=\int_{t_a}^{t_b}L\left(t,\bar{q}(t),{^{C}_{a}}\textsl{D}^{\alpha(t,\tau)}_{t}\bar{q}(t),
{^{C}_{t}}\textsl{D}^{\beta(t,\tau)}_{b}\bar{q}(t)\right)dt
\end{equation}
for any subinterval $[t_a,t_b]\subseteq [a,b]$.
\end{definition}

\begin{lemma}[Necessary condition of invariance]
\label{theorem:CondInv}
If functional \eqref{eq:Functional} is invariant under an $\varepsilon$-parameter group
of infinitesimal transformations \eqref{eq:transformation}, then
\begin{multline}
\label{eq:CondInv}
\partial_2 L\left\{q,\alpha,\beta\right\}(t) \xi(t,q(t))
+\sum\limits_{i=1}^n \partial_{i+2} L\left\{q,\alpha,\beta\right\}(t)
{^{C}_{a}}\textsl{D}^{\alpha_i(t,\tau)}_{t}\xi(t,q(t))\\
+\sum\limits_{i=1}^n \partial_{n+i+2} L\left\{q,\alpha,\beta\right\}(t)
{^{C}_{t}}\textsl{D}^{\beta_i(t,\tau)}_{b}\xi(t,q(t))=0.
\end{multline}
\end{lemma}

\begin{proof}
Since, by hypothesis, condition \eqref{eq:InvarianceCond}
is satisfied for any subinterval $[t_a,t_b]\subseteq [a,b]$, we have
\begin{equation}
\label{eq:2}
L\left(t,q(t),{^{C}_{a}}\textsl{D}^{\alpha(t,\tau)}_{t}q(t),
{^{C}_{t}}\textsl{D}^{\beta(t,\tau)}_{b}q(t)\right)
=L\left(t,\bar{q}(t),{^{C}_{a}}\textsl{D}^{\alpha(t,\tau)}_{t}\bar{q}(t),
{^{C}_{t}}\textsl{D}^{\beta(t,\tau)}_{b}\bar{q}(t)\right).
\end{equation}
Now, differentiating \eqref{eq:2} with respect to $\varepsilon$,
then putting $\varepsilon=0$, and applying definitions and properties
of variable order Caputo fractional derivatives, we obtain that
\begin{equation*}
\begin{split}
0&=\partial_2 L\left\{q,\alpha,\beta\right\}(t) \xi(t,q(t))\\
&\qquad +\sum\limits_{i=1}^n \partial_{i+2}
L\left\{q,\alpha,\beta\right\}(t)\frac{d}{d\varepsilon}\left[\int_a^t
\frac{1}{\Gamma(1-\alpha_i(t,\tau))}(t-\tau)^{-\alpha_i(t,\tau)}
\frac{d}{d\tau}\bar{q}(\tau)d\tau\right]_{\varepsilon=0}\\
&\qquad +\sum\limits_{i=1}^n\partial_{n+i+2}L\left\{q,\alpha,\beta\right\}(t)
\frac{d}{d\varepsilon}\left[\int_t^b
\frac{-1}{\Gamma(1-\beta_i(t,\tau))}(\tau-t)^{-\beta_i(t,\tau)}
\frac{d}{d\tau}\bar{q}(\tau)d\tau\right]_{\varepsilon=0}\\
&=\partial_2 L\left\{q,\alpha,\beta\right\}(t) \xi(t,q(t))
+\sum\limits_{i=1}^n \partial_{i+2} L\left\{q,\alpha,\beta\right\}(t)
{^{C}_{a}}\textsl{D}^{\alpha_i(t,\tau)}_{t}\xi(t,q(t))\\
&\qquad +\sum\limits_{i=1}^n \partial_{n+i+2} L\left\{q,\alpha,\beta\right\}(t)
{^{C}_{t}}\textsl{D}^{\beta_i(t,\tau)}_{b}\xi(t,q(t)).
\end{split}
\end{equation*}
\end{proof}

\begin{theorem}[Noether's theorem for variable order fractional variational problems]
\label{theorem:Noether}
If functional \eqref{eq:Functional} is invariant in the sense
of Definition~\ref{definition:InvariantFunct}, then
\begin{equation*}
\sum\limits_{i=1}^n\textsl{D}^{\alpha_i(t,\tau)}_{-}[\xi(t,q(t)),
\partial_{i+2} L\left\{q,\alpha,\beta\right\}(t)]\\
+\sum\limits_{i=1}^n\textsl{D}^{\beta_i(t,\tau)}_{+}[\xi(t,q(t)),
\partial_{n+i+2} L\left\{q,\alpha,\beta\right\}(t)]=0,
\quad t\in[a,b],
\end{equation*}
along all variable order fractional extremals $q(\cdot)$.
\end{theorem}

\begin{proof}
By Theorem~\ref{theorem:EL} we have
\begin{equation}
\label{eq:3}
\partial_2 L\left\{q,\alpha,\beta\right\}(t)
=-\sum\limits_{i=1}^n
{_{t}}\textsl{D}^{\alpha_i(t,\tau)}_{b}\partial_{i+2}
L\left\{q,\alpha,\beta\right\}(t)
-\sum\limits_{i=1}^n{_{a}}\textsl{D}^{\beta_i(t,\tau)}_{t}\partial_{n+i+2}
L\left\{q,\alpha,\beta\right\}(t).
\end{equation}
Substituting \eqref{eq:3} into \eqref{eq:CondInv}, we obtain
\begin{multline*}
-\sum\limits_{i=1}^n
\xi(t,q(t)){_{t}}\textsl{D}^{\alpha_i(t,\tau)}_{b}\partial_{i+2}
L\left\{q,\alpha,\beta\right\}(t)
-\sum\limits_{i=1}^n\xi(t,q(t)){_{a}}\textsl{D}^{\beta_i(t,\tau)}_{t}\partial_{n+i+2}
L\left\{q,\alpha,\beta\right\}(t)\\
+\sum\limits_{i=1}^n \partial_{i+2} L\left\{q,\alpha,\beta\right\}(t)
{^{C}_{a}}\textsl{D}^{\alpha_i(t,\tau)}_{t}\xi(t,q(t))
+\sum\limits_{i=1}^n \partial_{n+i+2} L\left\{q,\alpha,\beta\right\}(t)
{^{C}_{t}}\textsl{D}^{\beta_i(t,\tau)}_{b}\xi(t,q(t))=0.
\end{multline*}
Finally, by Definition~\ref{definition:OpProduct}, one has
\begin{equation*}
\sum\limits_{i=1}^n\textsl{D}^{\alpha_i(t,\tau)}_{-}[\xi(t,q(t)),
\partial_{i+2} L\left\{q,\alpha,\beta\right\}(t)]
+\sum\limits_{i=1}^n\textsl{D}^{\beta_i(t,\tau)}_{+}[\xi(t,q(t)),
\partial_{n+i+2} L\left\{q,\alpha,\beta\right\}(t)]=0.
\end{equation*}
\end{proof}


\section{An illustrative example}
\label{sec:ex}

Let $\alpha$ and $\beta$ be two functions such that $0<\alpha(t,\tau),\beta(t,\tau)<1-\frac{1}{l}$
for a certain natural number $l$ greater or equal than two. Consider the following problem:
\begin{equation}
\label{eq:example}
\begin{gathered}
\mathcal{J}[q]=\int_a^b L\left(t,{^{C}_{a}}\textsl{D}^{\alpha(t,\tau)}_{t}q(t),
{^{C}_{t}}\textsl{D}^{\beta(t,\tau)}_{b}q(t)\right) dt \longrightarrow \mbox{extremize}\\
q(a)=q_a \, , \quad q(b)=q_b.
\end{gathered}
\end{equation}
For the transformation
\begin{equation}
\label{eq:transformation2}
\bar{q}(t)=q(t)+\varepsilon c+o(\varepsilon),
\end{equation}
where $c$ is a constant, we have
\begin{equation*}
\int_{t_a}^{t_b}L\left(t,{^{C}_{a}}\textsl{D}^{\alpha(t,\tau)}_{t}q(t),
{^{C}_{t}}\textsl{D}^{\beta(t,\tau)}_{b}q(t)\right) dt
=\int_{t_a}^{t_b}L\left(t,{^{C}_{a}}\textsl{D}^{\alpha(t,\tau)}_{t}\bar{q}(t),
{^{C}_{t}}\textsl{D}^{\beta(t,\tau)}_{b}\bar{q}(t)\right)dt
\end{equation*}
for any $[t_a,t_b]\subseteq [a,b]$. Therefore, $\mathcal{J}[q]$ is invariant under
\eqref{eq:transformation2} and Noether's theorem (Theorem~\ref{theorem:Noether}) asserts that
\begin{equation*}
\textsl{D}^{\alpha(t,\tau)}_{-}\left[c,\partial_2L\left(t,{^{C}_{a}}\textsl{D}^{\alpha(t,\tau)}_{t}q(t),
{^{C}_{t}}\textsl{D}^{\beta(t,\tau)}_{b}q(t)\right)\right]\\
+\textsl{D}^{\beta(t,\tau)}_{+}\left[c,\partial_3L\left(t,{^{C}_{a}}\textsl{D}^{\alpha(t,\tau)}_{t}q(t),
{^{C}_{t}}\textsl{D}^{\beta(t,\tau)}_{b}q(t)\right)\right]=0
\end{equation*}
along any extremal $q(t)$ of \eqref{eq:example}.


\section{Conclusion}
\label{sec:conc}

Noether's symmetry theorem, establishing that variational invariance imply
conservation laws, is one of the most important results in physics
and the calculus of variations. Such conservation laws are,
however, only valid for conservative systems.
Nonconservative forces, like friction, remove energy from the
systems and, as a consequence, Noether's conservation
laws cease to be valid. In order to cope with dissipative forces
that do not store energy, one possibility is to use fractional calculus
\cite{CD:Riewe:1996,CD:Riewe:1997}. The study of problems of the calculus
of variations with fractional derivatives of variable order
is, however, a rather recent subject, and available results reduce
to those available at \cite{Atanackovic1,MyID:241,Tatiana:IDOTA2011}.
Here we obtain Euler--Lagrange optimality conditions
for variational problems with fractional derivatives of incommensurate variable order
and, for such fractional Euler--Lagrange extremals,
we provide a notion of conservation law
and prove a version of Noether's theorem.

The variable order calculus of variations is underdeveloped and much remains to be done.
We mention here one important question that remains open. The Noether type theorem we obtain here
only includes the terms related to ``conservation of momentum''.
It would be important to obtain a more general form of the Noether theorem
with ``energy'' terms.


\section*{Acknowledgements}

Work supported by {\it FEDER} funds through
{\it COMPETE} --- Operational Programme Factors of Competitiveness
(``Programa Operacional Factores de Competitividade'')
and by Portuguese funds through the
{\it Center for Research and Development
in Mathematics and Applications} (University of Aveiro)
and the Portuguese Foundation for Science and Technology
(``FCT --- Funda\c{c}\~{a}o para a Ci\^{e}ncia e a Tecnologia''),
within project PEst-C/MAT/UI4106/2011
with COMPETE number FCOMP-01-0124-FEDER-022690.
Odzijewicz was also supported by FCT through the Ph.D. fellowship
SFRH/BD/33865/2009; Malinowska through the project ``Information Platform TEWI''
co-financed by the European Union under the European Regional Development Fund;
and Torres by FCT, through the project PTDC/MAT/113470/2009,
and by the European Union Seventh Framework Programme FP7-PEOPLE-2010-ITN
under grant 64735-SADCO.



\end{document}